\documentclass{article}

\usepackage[latin1]{inputenc}
\usepackage[english]{babel}
\usepackage{times}
\usepackage{epsfig}
\usepackage{amsfonts}
\usepackage{amsthm}
\usepackage{float}
\usepackage{amsmath}
\usepackage{amssymb}
\usepackage{latexsym}
\usepackage{graphicx, color}
\usepackage{caption}
\usepackage{subcaption}
\usepackage[round]{natbib}
\usepackage{bm}
\usepackage{textcomp}
 \usepackage{paralist}
\usepackage{tikz}
\usetikzlibrary{arrows,automata,positioning}

\newcommand{\rnc}{\renewcommand}
\newcommand{\nc}{\newcommand}
\newcommand{\mrm}{\mathrm}
\renewcommand{\b}{\textbf}
\newcommand{\bs}{\boldsymbol}
\nc{\mb}{\mathbb}
\nc{\mac}{\mathcal}
\nc{\E}{\mb{E}}
\nc{\N}{\mb{N}}
\nc{\R}{\mb{R}}
\nc{\Q}{\mb{Q}}
\rnc{\P}{\mrm P}
\rnc{\d}{\mrm d}
\nc{\C}{\mac{C}}
\nc{\D}{\mac{D}}
\nc{\B}{\mac{B}}
\nc{\oPo}{\stackrel{p}{\longrightarrow}}
\nc{\oWo}{\stackrel{w}{\longrightarrow}}
\nc{\oDo}{\stackrel{d}{\longrightarrow}}

\relax 
\allowdisplaybreaks[3]

\relax

\newcommand{\dist}{\mbox{$\, \stackrel{d}{\longrightarrow} \,$}}
\newcommand{\prob}{\mbox{$\, \stackrel{p}{\longrightarrow} \,$}}

\renewcommand{\qed}{\mbox{ } \hfill $\Box$\\ }

\newcommand{\bay}{\begin{array}}
\newcommand{\eay}{\end{array}}

\newcommand{\bqa}{\begin{eqnarray*}}
\newcommand{\eqa}{\end{eqnarray*}}

\newcommand{\bee}{\begin{eqnarray*}}
\newcommand{\eee}{\end{eqnarray*}}

\newcommand{\bea}{\begin{eqnarray*}}
\newcommand{\eea}{\end{eqnarray*}}

\newcommand{\bqan}{\begin{eqnarray}}
\newcommand{\eqan}{\end{eqnarray}}

\newcommand{\be}{\begin{eqnarray}}
\newcommand{\ee}{\end{eqnarray}}

\newcommand{\bit}{\begin{itemize}}
\newcommand{\eit}{\end{itemize}}

\newcommand{\ben}{\begin{enumerate}}
\newcommand{\een}{\end{enumerate}}

\newcommand{\beq}{\begin{equation}}
\newcommand{\eeq}{\end{equation}}

\newcommand{\bdes}{\begin{description}}
\newcommand{\edes}{\end{description}}

\newcommand{\btb}{\begin{tabular}}
\newcommand{\etb}{\end{tabular}}

\newcommand{\bcen}{\begin{center}}
\newcommand{\ecen}{\end{center}}

\newcommand{\bmp}{\begin{minipage}}
\newcommand{\emp}{\end{minipage}}

\newcommand{\var}{\operatorname{{\it Var}}}

\newcommand{\vlambda}{\mbox{\boldmath $\lambda$}}

\newcommand{\vN}{\boldsymbol{N}}

\newtheorem{definition}{{\sc Definition}\sc}[section]
\newcommand{\bdefi}{\begin{definition}}
\newcommand{\edefi}{\end{definition}}

\newtheorem{appropr}[definition]{{\sc Approximation Procedure}\sc}
\newcommand{\bappr}{\begin{appropr}}
\newcommand{\eappr}{\end{appropr}}

\newtheorem{bedi}[definition]{{\sc Condition}\sc}
\newcommand{\bbd}{\begin{bedi}}
\newcommand{\ebd}{\end{bedi}}

\newtheorem{bedin}[definition]{{\sc Conditions}\sc}
\newcommand{\bbdn}{\begin{bedin}}
\newcommand{\ebdn}{\end{bedin}}

\newtheorem{corollary}[definition]{{\sc Corollary}\sc}
\newcommand{\bco}{\begin{corollary}}
\newcommand{\eco}{\end{corollary}}

\newtheorem{lemma}[definition]{{\sc Lemma}\sc}
\newcommand{\blem}{\begin{lemma}}
\newcommand{\elem}{\end{lemma}}

\newtheorem{proposition}[definition]{{\sc Proposition}\sc}
\newcommand{\bpro}{\begin{proposition}}
\newcommand{\epro}{\end{proposition}}

\newtheorem{satz}[definition]{{\sc Theorem}\sc}
\newcommand{\bsa}{\begin{satz}}
\newcommand{\esa}{\end{satz}}

\newtheorem{assumption}[definition]{{\sc Assumption}\sc}
\newcommand{\bas}{\begin{assumption}}
\newcommand{\eas}{\end{assumption}}

\newtheorem{assumptions}[definition]{{\sc Assumptions}\sc}
\newcommand{\bass}{\begin{assumptions}}
\newcommand{\eass}{\end{assumptions}}

\newtheorem{abb}{{\sc Figure}\sc}
\newcommand{\babb}{\begin{abb}}
\newcommand{\eabb}{\end{abb}}

\newenvironment{remark}{\begin{rmk}\sl}{\end{rmk}}
\newtheorem{rmk}{{\sc Remark}\sc}[section]
\newcommand{\brem}{\begin{remark}}
\newcommand{\erem}{\end{remark}}

\newenvironment{remarks}{\begin{rmks}\sl}{\end{rmks}}
\newtheorem{rmks}{{\sc Remarks}\sc}[section]
\newcommand{\brems}{\begin{remarks}}
\newcommand{\erems}{\end{remarks}}

\newenvironment{example}{\begin{exmp}\rm}{\end{exmp}}
\newtheorem{exmp}{{\sc Example}\sc}[section]
\newcommand{\bbsp}{\begin{example}}
\newcommand{\ebsp}{\end{example}}
\newcommand{\bexa}{\begin{example}}
\newcommand{\eexa}{\end{example}}

\newtheorem{model}{{\sc Model}\sc}[section]
\newcommand{\bmdl}{\begin{model}}
\newcommand{\emdl}{\end{model}}

\newtheorem{scheme}{{\sc Scheme}\sc}[section]
\newcommand{\bscm}{\begin{scheme}}
\newcommand{\escm}{\end{scheme}}

\newenvironment{tabelle}{\begin{tabl}\sl}{\end{tabl}}
\newtheorem{tabl}{{\sc Table}\sc}
\newcommand{\btab}{\begin{tabelle}}
\newcommand{\etab}{\end{tabelle}}

\newenvironment{exercise}{\begin{exc}\sl}{\end{exc}}
\newtheorem{exc}{Exercise}[section]
\newcommand{\bexe}{\begin{exercise}}
\newcommand{\eexe}{\end{exercise}}

\numberwithin{equation}{section}
\rnc{\var}{\textnormal{var}}

\newtheorem{thm}{Theorem}
\newtheorem{rem}{Remark}
\newtheorem{exs}{Examples}

\begin{document}

    
\title{\bf Non-strange Weird Resampling for Complex Survival Data}
    
\author{Dennis Dobler$^{*,1}$ and Jan Beyersmann$^{*}$ and Markus Pauly$^{*}$ 
}
\maketitle

\begin{abstract}

  This paper introduces the new data-dependent multiplier bootstrap for
  non-parametric analysis of survival data, possibly subject to competing
  risks. The new resampling procedure includes both the general wild bootstrap
  and the weird bootstrap as special cases. The data may be subject to
  independent right-censoring and left-truncation. We rigorously prove
  asymptotic correctness which has in particular been pending for the weird
  bootstrap. As a consequence, pointwise as well as time-simultaneous inference procedures
  for, amongst others, the classical survival setting are deduced.
  We report simulation results and a real data analysis of the
  cumulative cardiovascular event probability. The simulation results suggest
  that both the weird bootstrap and use of non-standard multipliers in the
  wild bootstrap may perform preferably.
  
\end{abstract}


\vfill
\vfill

\noindent${}^{*}$ {University of Ulm, Institute of Statistics, Germany}\\
\noindent${}^{1}$ {Corresponding author's email: dennis.dobler@uni-ulm.de}

\newpage

\section{Introduction}
Non-parametric inference for time-to-event data is often hindered by
asymptotic non-pivotality. The Nelson-Aalen (NAE) and Aalen-Johansen
estimators (AJE) converge weakly to Gaussian processes on a Skorohod space,
but their covariance functions depend on unknown quantities. This problem is
attacked by plug-in of estimates and, in the absence of competing risks,
transformation of the limit distribution towards a Brownian bridge
\citep[e.g.,\ ][Section~IV.1.3.2]{abgk93}. The latter approach, however, fails
if interest lies in the cumulative event probability of a competing risk, and
resampling techniques are needed. Moreover, even if asymptotic pivotal approximations are
available, resampling is well known to often perform advantageously in small
samples.

In an i.i.d.\ setting, resampling typically uses the classical bootstrap of
\cite{efron79}, extended to the Kaplan-Meier estimator by repeatedly taking
random samples with replacement from the randomly censored observations
\citep{efron81}. Theoretical justifications were provided by \cite{akritas86}
and \cite{lo86}. The latter authors also suggested that their method of proof
extends to the situation of competing risks. In an article about weak
convergence for quantile processes and their bootstrap versions, \cite{doss92}
briefly discussed resampling inference for the latent failure time of a
competing risk.

Another popular resampling method traces back to \cite{lin97}, see also the
textbook treatment by \cite{Mart:Sche:dyna:2006}. Lin's idea was to consider
the martingale representation of the AJE which originates from the Doob-Meyer
decomposition for counting processes and does not necessarily require an
i.i.d.\ setup \citep[e.g.,\ ][]{abgk93}. Lin suggested to replace the
martingale increments by the increments of the observed counting processes,
reweighted by standard normal variates. The approach has recently been
recognized as a special case of the wild bootstrap \citep[e.g.,\
][]{beyersmann13}, where the weights are required to have mean zero and
variance one, but need not be normal.

Yet another, earlier suggestion for a general resampling procedure for
time-to-event data is the \emph{weird bootstrap} due to
\citet[Sec.~IV.1.4]{abgk93}, but it has slightly fallen into
oblivion. Andersen et al.\ formulated their ideas for approximating the
empirical distribution of the standardized NAE, say
$V_n=\sqrt{n}(\widehat{A}-A)$. Since the counting process increments ${\rm d}N(t)$
that enter the NAE~$\widehat{A}(t)$ have the same conditional variance as
$B(Y(t),{\rm d}A(t))$-distributed binomial random variables, given the risk
set~$Y(t)$ at $t-$, it seems natural to consider a corresponding {\it weird
  jump process} $N^*$ with independent and
$B(Y(t),{\rm d}\widehat{A}(t))$-distributed increments at the jump times of $N$.
This results in a so-called weird bootstrap NAE version $\widehat
V_n=\sqrt{n}(\widehat{A}^*-\widehat{A}) =\sqrt{n} \int 1/{Y} ({\rm d}N^*-{\rm d}N)$. At
first sight, this bootstrap is `weird' in that the number at risk is not
changed in the bootstrap step and thus each individual may cause several
simulated events. At second sight, however, the weird bootstrap is a very natural approach as discussed in Section~\ref{sec: dis} below. 
Andersen et al.\ sketched a theoretical
justification for weird bootstrapping the NAE, but --- as also \citet[p.~38]{freitag00} pointed
out --- a rigorous proof has not been given.

Although the weird bootstrap has been implemented in the functions
\emph{censboot} and \emph{coxreg} of the \verb=R= packages \verb=boot= and
\verb=eha=, respectively, \citep[for the latter, see Appendix D.2
of][]{brostrom12}, Efron's bootstrap and the wild bootstrap with standard
normal weights are the most popular resampling schemes in the survival
literature. Exceptions using the weird bootstrap are \cite{dudek08} and
\cite{fledelius04}. Dudek et al.\ empirically found superiority of some weird
bootstrap confidence bands for the cumulative hazard rate compared to using
Efron's approach. Fledelius et al.\ studied residual lifetimes and proposed
the weird bootstrap for a kernel density estimator of the hazard rate. These
authors accounted for both the age of individuals under study and calendar
time, leading to a two-dimensional time parameter.  Weak convergence is shown
for arbitrary single points of time, but not time-simultaneously, yielding
confidence intervals rather than confidence bands. For another brief textbook treatment of the weird bootstrap, see also \citet[Sections 3.5 and 7.3]{davison97}.

The aim of this paper is to introduce and rigorously justify a new resampling
procedure, the \emph{data-dependent multiplier bootstrap} (DDMB) for
non-parametric analysis of survival data, possibly subject to competing risks,
that includes both the general wild bootstrap and the weird bootstrap as
special cases. The data are assumed to be subject to independent
right-censoring and left-truncation, but a strict i.i.d.\ setup is not
required. (In fact, all that is really needed is the multiplicative intensity
model, see, e.g., \citet[Section 3.1.2]{aalen08}).
As a byproduct, our development includes a rigorous proof for the original weird bootstrap. In
contrast to the classical wild bootstrap, the new procedure allows for both
non-i.i.d.\ weights and data-dependent weights. Expressing the weird
bootstrap as a DDMB, the corresponding multipliers approximately correspond to
independent $Poi(1)$ variates for large numbers of individuals at risk,
arriving at a special wild bootstrap version as studied
in~\cite{beyersmann13}.

For ease of presentation, we formulate our developments for the AJE of the
cumulative incidence functions (CIFs) in a competing risks setting. This
includes the standard survival scenario in which there is only one event
type. For applications of the DDMB, we study the testing problem of
equality of CIFs from two independent groups \citep[see also,
e.g.,][]{bajorunaite07, bajorunaite08,dobler14a, dobler14b}, and we construct
asymptotically valid confidence bands for CIFs \citep[see also,
e.g.,][]{lin97,beyersmann13}.

This article is organized as follows. Section~\ref{sec:model} recaps the
properties of the competing risks model under consideration, the quantity of
interest (the CIF) and its canonical estimators. The DDMB and its special forms are
introduced and analyzed in Section~\ref{sec:ddmb} and applications for the
two-sample testing problem as well as for
time-simultaneous confidence bands are given in Section~\ref{sec:tests}.  Small
sample performance of confidence bands is assessed in a simulation study in
Section~\ref{sec: simus}. The simulation setup has been chosen similar to a
randomized clinical trial on cardiovascular events in diabetes patients
\citep{wanner05}, and real data from this trial are then analyzed in
Section~\ref{sec: example}.  Finally, we give some concluding remarks in
Section~\ref{sec: dis}. All proofs are deferred to the Appendix.

\section{Notation, Model and Estimators}
\label{sec:model}

The ordinary survival setup is generalized to a {\it competing risks process} $(X_t)_{t\geq 0}$ 
with $m\in\mathbb{N}$ competing risks.
This is a non-homogeneous Markov process with state space $\{0,1,\dots,m\}$
and initial state $0$, i.e., ${P}(X_0=0)=1$.
All other states~$1,\dots,m$ represent absorbing competing risks. 
For ease of notation we only discuss the case $m=2$ with two absorbing states since generalizations to $m\geq3$ are obvious. 
The {\it event time} ${T}=\inf\{t>0 \mid\ X_t\neq 0\}$ is assumed to be finite a.s.. 
The process behaviour is regulated by the transition intensities (or cause-specific hazard functions) 
between states $0$ and $j=1,2,$ denoted by
\begin{equation}
\label{eq:alpha}
  \alpha_{j}(t) = \lim_{\Delta t \searrow 0}\frac{{P}({T}\in [t, t + \Delta t), X_{{T}}=j\,\mid\, {T} \geq t)}{\Delta t},\ \, j=1,2.
\end{equation}
Throughout we assume that  $\alpha_{1}$ and $\alpha_{2}$ exist. 
One is often interested in the development of the competing risks process in time on a given compact interval $I \subseteq [0,\tau)$. 
Here $\tau$ is an arbitrary terminal time such that $\tau \le \sup\{u: \int_0^u(\alpha_{1}(s)+\alpha_{2}(s)) {\rm d} s < \infty \}$ 
whence ${P}({T}> \cdot)>0$ on $[0,\tau)$. 
For a detailed motivation and more practical examples for occurrences of competing risks designs 
we refer to \cite{abgk93}, \cite{Allignol_10} as well as \cite{beyersmann12}.\\

For $n$ independent replicates of the competing risks process, i.e. $n$ individuals under study, we now consider the associated bivariate counting process $\vN=(N_1,N_2)$. 
Here $N_j= \sum_{i=1}^n N_{j;i}, j=1,2,$ with
\begin{equation}
\label{eq:n}
  N_{j;i}(t)=\mathbf{1}\left(\text{ Subject i has an observed } (0 \rightarrow j) \text{ -- transition in } [0,t]\right),
\end{equation}
counts the number of observed transitions into state~$j \in \{1,2\}$, where $\mathbf{1}(\cdot)$ denotes the indicator function. 
As usual, it is postulated that the processes $N_1$ and $N_2$ are c\`adl\`ag and do not jump simultaneously. 
Moreover, we assume that $\vN$ fulfills the multiplicative intensity model given in~\cite{abgk93},
i.e., its intensity process $\vlambda=(\lambda_1,\lambda_2)$ is given by
\begin{equation}
\label{eq:mult_inten}
  \lambda_j(t) =  Y(t) \alpha_{j}(t),\quad j=1,2.
\end{equation}
Here $Y= \sum_{i=1}^n Y_i$ and 
\begin{equation}
 \label{eq:y}
Y_i(t)= \mathbf{1}\left( \text{ Subject i is in state } 0  \text{ at time } t- \right),
\end{equation}
i.e. $Y$ counts the number at risk immediately before time $t$. 
It is worth to note that the multiplicative intensity model holds, for instance, in the context of independent right-censoring or left-truncation; 
see Chapter~III and IV in~\cite{abgk93}. 
Moreover, even different censoring distributions are possible; see Example~IV.1.6 in the same textbook.
For the explicit modelling of these incomplete observations in various settings we again refer to the monograph of~\cite{abgk93}.
{\color{black}
Other kinds of multiplicative intensity models are also conceivable in combination with the present theory
such that the number at risk process $Y$ may be replaced with a more general predictable process depending on the model that describes the data;
see the examples in Section 3.1.2 of \cite{aalen08}.}

We are now interested in the cumulative incidence functions (CIFs), or sub-distribution functions, given by
$$
  F_j(t) = {P}({T}\leq t, X_{ T}=j)=\int_{0}^{t}{P}({T}> u-)\alpha_{j}(u){\rm d}u, \quad j=1, 2,
$$
which depend only on the cause-specific transition intensities $\alpha_1$ and $\alpha_2$. 
The corresponding sub-survival functions are denoted $S_j(t)=1-F_j(t)$, $j=1,2$ and
the Aalen-Johansen estimators for the CIFs are
\begin{equation}
\label{eq:AJ}
\widehat{F}_j(t) = \int_{0}^{t}\frac{\widehat{{P}}(T>u-)}{Y(u)} J(u) \ {\rm d}N_{j}(u), \quad j=1,2.
\end{equation}
Here $J(u)=\mathbf{1}(Y(u)>0)$ (such that $\frac00 := 0$) and $\widehat{{P}}(T>u)$ denotes the Kaplan-Meier estimator. 
As above we denote the estimator for the sub-survival function by $\widehat{S}_j(t) =  1 - \widehat{F}_j(t)$. 
Note that the usual survival scenario is obtained by letting $\alpha_2 \equiv N_2 \equiv 0$
so that $\widehat S_1(t)$ reduces to the Kaplan-Meier estimator.

Simultaneous confidence bands for a CIF, say $F_1$, are typically based on the Aalen-Johansen process via
\begin{equation}
\label{eq:w}
  W_n(\cdot) = n^{1/2}\{\widehat{F}_1(\cdot)-{F}_1(\cdot)\}.
\end{equation}
Under the following throughout assumed regularity assumption (where $y: I \rightarrow \mathbb{R}$ is a deterministic function)
\bqan
\label{eq:cond_beyersmann}
 \sup_{u \in I} \left| \frac{Y(u)}{n} - y(u) \right| \prob 0 \quad \text{with} \quad \inf_{u \in I} y(u) > 0,
\eqan
$W_n$ converges in distribution on the Skorohod space $\mathcal{D}(I)$ to a zero-mean Gaussian process $U$; 
see e.g. Theorem~IV.4.2 in~\cite{abgk93}. 
Here and throughout the paper, ``$\prob$'' denotes convergence in probability, whereas ``$\dist$'' stands for convergence in distribution as $n \rightarrow \infty$. 
In particular, we have 
\begin{equation}\label{eq:weakconv}
 W_n \dist U \quad \text{on}\quad \mathcal{D}(I),
\end{equation}
where $U$ is a  zero-mean Gaussian process with covariance function given by
\begin{eqnarray}\label{eq:zeta}
 \zeta({s_1},{s_2})&=&\int_0^{s_1} \frac{\{S_2(u)-F_1({s_2})\}\{S_2(u)-F_1(s_1)\} \alpha_{1}(u)}{y(u)} {\rm d} u\nonumber\\
   &+& \int_0^{s_1} \frac{\{F_1(u)-F_1({s_2})\}\{F_1(u)-F_1({s_1})\} \alpha_{2}(u)}{y(u)} {\rm d} u\label{eq:cov}
\end{eqnarray}
for $s_1<s_2$. This martingale-based weak convergence result follows from the representation
\begin{eqnarray}\label{eq:martingale rep}
  W_n(t) =  \sqrt{n} \sum_{i=1}^{n}\Big( \int_{0}^{t}\frac{S_2(u) - F_1(t)}{{Y}(u)}{\rm d}{M}_{1; i}(u)
  {+ \int_{0}^{t}\frac{F_1(u) - F_1(t)}{{Y}(u)}}{\rm d}{M}_{2;i}(u)\Big)+o_{p}(1),
\end{eqnarray}
where for $1\leq i\leq n, j=1,2,$
\begin{equation*}
  M_{j;i}(s) = N_{j;i}(s) - \int_0^s Y_i(u) \alpha_{j}(u)\, {\rm d}u
\end{equation*}
are square integrable martingales. 
For ease of notation the dependency on $n$ and the appearance of the indicator $J(u)$ is suppressed in both integrals in (\ref{eq:martingale rep}). 
The convergence in~\eqref{eq:weakconv} finally follows from \eqref{eq:cond_beyersmann} 
in combination with Rebolledo's martingale central limit theorem \citep[see][Theorem~II.5.1]{abgk93}.
Note, that the main assumption \eqref{eq:cond_beyersmann} is satisfied in most relevant situations, 
e.g., for right-censored and left-truncated or even filtered data; see Sections III and IV in~\cite{abgk93}.

Since the covariance function $\zeta$ is unknown and lacks independent increments, 
resampling techniques are essential for approximating the distribution $\mathcal{L}(W_n)$ of $W_n$.
Therefore, we introduce a general DDMB method.

\section{The Data-Dependent Multiplier Bootstrap}
\label{sec:ddmb}

The last mentioned covariance problem is typically attacked using a computationally convenient resampling technique which is due to 
\cite{lin93} and \cite{lin94, lin97}. 
Their idea is to replace the unobservable martingales ${M}_{j; i}$ in the representation \eqref{eq:martingale rep}
with i.i.d. standard normal variates $G_{j;i},\, i \in \mathbb{N},\,j=1,2$, 
(which are independent from the data) times the observable counting processes ${N}_{j; i}$. 
Moreover, all remaining unknown quantities in \eqref{eq:martingale rep} are replaced with their estimators.
This leads to the following resampling version of $W_n$ according to \cite{lin97}:
$$
\widehat{W}_n(t)  =  \sqrt{n} \sum_{i=1}^n \big( \int_0^t \frac{G_{1;i} (\widehat{S}_2(u) - \widehat{F}_1(t))}{Y(u)} \ {\rm d}N_{1;i}(u)
		     + \int_0^t \frac{ G_{2;i}(\widehat{F}_1(u) - \widehat{F}_1(t))}{Y(u)} \ {\rm d}N_{2;i}(u)\big);
$$
see also~\cite{beyersmann13}
where the validity of this approach is proven for the even
more general wild bootstrap with i.i.d. zero-mean random variables $G_{j;i},\, i \in \mathbb{N},\,j=1,2,$ with variance $1$ and finite fourth moments. 
This means that the conditional distribution of $\widehat{W}_n$ asymptotically coincides with that of $W_n$. 
Hence its law may be approximated via a large number of realizations, repeatedly generating i.i.d. multipliers $G_{j;i}$. 
In the following we show how to generalize this method to the case of data-dependent multiplier weights $(D_{n;i})_{i,n}$ which are only supposed to be conditionally independent given the data. 
An advantage of this approach is the possibility to weight the individual subjects in diverse ways.
For example, certain preferences (e.g. depending on the time under study) can be taken into account, specifically arriving at the weird bootstrap from~\cite{abgk93};
see Examples~\ref{ex: proc} below.
To this end we rewrite $\widehat{W}_n$ as
\begin{equation}\label{eq:wildbstat}
  \widehat{W}_n(t)  =  \sqrt{n} \sum_{i=1}^{n} \Big(G_{1;i}X_{n;1;i}(t) + G_{2;i}X_{n;2;i}(t)\Big) = \sqrt{n} \sum_{i=1}^{2n} G_{i} Z_{2n;i}(t),
\end{equation}
where for $s \in I$ and $i = 1, \dots, n$
\bqa
   X_{n;1;i}(s) = \int_0^s \frac{\widehat S_2(u) - \widehat F_1(s)}{Y(u)} J(u) \ {\rm d} N_{1;i}(u) , \qquad
   X_{n;2;i}(s) = \int_0^s \frac{\widehat F_1(u) - \widehat F_1(s)}{Y(u)} J(u) \ {\rm d} N_{2;i}(u) ,
\eqa
$G_i = G_{1;i}\mathbf 1(i \leq n) + G_{2;i-n} \mathbf 1(i > n)$ and $Z_{2n;i} := X_{n;1;i} \mathbf 1(i \leq n) + X_{n;2;i-n} \mathbf 1(i > n)$.
That is, we obtain a linear weighted representation as in \cite{dobler14a}.

Now replacing the i.i.d. weights $G_i$ in \eqref{eq:wildbstat}  with data-dependent multipliers $(D_{2n;i})_{i,n}$, 
we arrive at the so-called DDMB version of the normalized Aalen-Johansen estimator
\begin{equation}\label{eq:wildbstat b}                                                                                                                                                                      
\widehat W_n^D = \sqrt{n} \sum_{i=1}^{2n} D_{2n;i} Z_{2n;i}.
\end{equation}
These bootstrap weights also need to fulfill regularity conditions concerning their conditional moments in order to induce conditional finite-dimensional convergence and tightness.
In particular, the Conditions~\eqref{eq:conv_wbs_1}-\eqref{eq:conv_wbs_6} below guarantee the validity of this approach, i.e. 
the weak convergence on the Skorohod space $\D(I)$ to the Gaussian process $U$. 
Its proof depends on an application of Theorem~13.5 of~\cite{billingsley99} and is split up into two parts;
see Lemma~\ref{lemma:fidi_conv_general_wild_bs} and~\ref{lemma:tightness_wbs} in the Appendix. 

{\color{black}
For the purpose of applying the theory developed in this paper,
we again stress that only the multiplicative intensity model~\eqref{eq:mult_inten} and Condition~\eqref{eq:cond_beyersmann} are required.
Hence, all available information is given by the processes $t \mapsto (Y_i(t), N_{1;i}(t), N_{2;i}(t))$ for all $i=1, \dots, n$
and the $\sigma$-field containing (at least) all this information is denoted $\mac{A}_n$.
This scenario includes, for example, independent left-truncation and right-censoring
in which case we can equivalently write $\mac{A}_n = \sigma(L_i, \tilde T_i, \delta_i: i = 1, \dots, n)$.
Here $L_i$ denotes the entry time into the study for individual $i$, $\tilde T_i > L_i$ is its event or censoring time, whichever comes first, 
and $\delta_i$ indicates the type of event in case of $\delta_i \in \{1,2\}$
or a censored observation for $\delta_i = 0$.
}

Further, the product measure of (conditional) distributions $P_i, i=1,\dots,n,$ is indicated by $\bigotimes_{i=1}^n P_i$
and the notation $V_n \in \mac{O}_p(r_n)$ describes the following boundedness property in probability: 
there exists a constant $C > 0$ such that $r_n^{-1} | V_n | + o_p(1) \leq C$ a.s. for all $n$.
\begin{thm}
\label{thm:ddmb_cifs}
  Suppose that \eqref{eq:cond_beyersmann} holds and that the DDMB weights 
  $(D_{2n;i})_{i,n}$ fulfill
  \begin{align}
   \label{eq:conv_wbs_1} & \max_{1 \leq i \leq 2n} | \mu_{n;i} | := \max_{1 \leq i \leq 2n} | \E[ D_{2n;i} | \mac{A}_n ] | \in o_p(n^{-1/2}), \\
   \label{eq:conv_wbs_2} & \max_{1 \leq i \leq 2n} | \sigma_{n;i}^2 - 1 | := \max_{1 \leq i \leq 2n} | \var(D_{2n;i} | \mac{A}_n) - 1 | \in o_p(1), \\
   \label{eq:conv_wbs_4} & \max_{1 \leq i \leq 2n} \E[ D^4_{2n;i} |\mac{A}_n] \in \mac{O}_p(n), \\
   \label{eq:conv_wbs_5} & \mac{L} ( D_{2n;i}, i = 1,\dots,2n \;|\; \mac{A}_n )
    = \bigotimes_{i = 1,\dots,2n} \mac{L} (  D_{2n;i} \; | \; \mac{A}_n ).
  \end{align}
 If in addition $(D_{2n;i})_{i,n}$ satisfy the following conditional Lindeberg condition in probability given $\mac{A}_n$
   \begin{align}
   \label{eq:conv_wbs_6}
    \sum_{i=1}^n \E \Big[ \frac{(D_{2n;i} - \mu_{n;i} )^2}{\sum_{j=1}^n \sigma_{n;j}^2} 
      \textnormal{\b 1} \Big( \frac{(D_{2n;i} - \mu_{n;i})^2}{\sum_{j=1}^n \sigma_{n;j}^2} > \varepsilon \Big) \Big| \mac{A}_n \Big] \oPo 0 
      \qquad 
      \text{for all }  \varepsilon > 0,
  \end{align}
  then the DDMB version of the AJE 
  converges in distribution on the Skorohod space $\D(I)$ 
  to the Gaussian process $U$ given in \eqref{eq:weakconv}.
  I.e., given $\mac{A}_n$ we have
  \begin{align*}
    \widehat W_n^D = \sqrt{n} \sum_{i=1}^{2n} D_{2n;i} Z_{2n;i} \oDo U \quad \text{in probability}.
  \end{align*}
 \end{thm}
\begin{rem}
\label{rem:thm_ddmb}
  \text{ }{\\}
  \textnormal{(a)} The involved Lindeberg condition~\eqref{eq:conv_wbs_6} is implied by \eqref{eq:conv_wbs_2} 
  combined with
  $$ \max_{1 \leq i \leq 2n} \E[ D^4_{2n;i} |\mac{A}_n] \in \mac{O}_p(n^{\varepsilon}) 
  \qquad \text{for some } \varepsilon \in (0,1),$$ 
  instead of Condition~\eqref{eq:conv_wbs_4}
  since the multipliers $D_{2n;i}$ then induce a conditional Lyapunov central limit theorem.\\
  \textnormal{(b)} In Theorem~\ref{thm:ddmb_cifs} it is important that the DDMB weights are not influenced by the data in the limit. For example, the asymptotic variances 
$\sigma_{i;n}^2$ should be $1$ regardless of the actual data. In this way DDMB weights and (non-identically distributed) wild bootstrap weights
    are seen to be equivalent asymptotically.
\end{rem}

The conditions of Theorem~\ref{thm:ddmb_cifs} are satisfied for the following resampling schemes.
\begin{exs}\label{ex: proc}
 \textnormal{(a)} The wild bootstrap as in~\cite{beyersmann13} with i.i.d. multipliers $D_{2n;i}= G_i$ having mean zero, 
 variance $1$ and finite fourth moment falls under our approach.\\
 \textnormal{(b)} 
  As special cases of \textnormal{(a)} we obtain the resampling technique of~\cite{lin97} with i.i.d. standard normal weights $G_i$ as well as the Poisson-wild bootstrap with data-independent weights 
$D_{2n;i}\stackrel{i.i.d.}{\sim} Poi(1) - 1$.\\
 \textnormal{(c)} Moreover, even a wild bootstrap with non-identically distributed random variates $G_i$, all having mean zero, variance $1$ and finite fourth moment, is covered. \\
\textnormal{(d)} Another example is the \emph{weird bootstrap} of~\cite[Section~IV.1.4]{abgk93}. 
    For simplicity, we abbreviate $(M_i)_i := (M_{j;i})_{i,j}$ and $(N_i)_i := (N_{j;i})_{i,j}$.
    Applying the procedure from above,
    we replace the individual- and transition-specific martingales $\d M_i$ with $B_i \d N_i$.
    Here the random variable $B_i$ is given by
    \begin{align}
    \label{eq:weird_weight}
      B_i = \B \Big(Y(\tilde T_i), \frac{1}{Y(\tilde T_i)} \Big) - 1
    \end{align}
    with $\tilde T_i$ as above 
    and all binomially-$\B (Y(\tilde T_i), 1/Y(\tilde T_i))$ distributed random variables
    are assumed to be independent given the data.
    Note that the subtraction of $1$ in~\eqref{eq:weird_weight} corresponds to a centering at $\sum_{i=1}^{2n} Z_{2n;i}$;
    in~\cite{abgk93} this is done by subtracting the Nelson-Aalen estimator.
    The centering by 1 can also be deemed as $\E[B_i | \mac{A}_n ] = 0$;
    note here that $Y(\tilde T_i) > 0$ for all $i$.
    Further, the variances are given by $\var(B_i | \mac{A}_n ) = 1 - Y(\tilde T_i)^{-1} \oPo 1$, cf. Condition~\eqref{eq:cond_beyersmann}. 
    This again shows the close connection between weird and wild bootstrap (with Poisson weights).
    However, 
    these binomial objects are in general unconditionally dependent of the data
    since the above parameters depend on $Y$ and $\tilde T_i$.\\
\textnormal{(e)} Moreover, other data-dependent multipliers that put different weights on observations depending on their time under study are conceivable. 
  A special example is given at the end of the next section. 
\end{exs}
\section{Deduced Inference Procedures}
\label{sec:tests}
In this section we exemplify some inferential applications of the developed methodology.
Throughout, let $I = [t_1,t_2] \subseteq [0,\tau)$ again be any compact interval.

\subsection{Simultaneous Confidence Bands, One-Sample Tests and Confidence Intervals}

Following~\cite{lin97} and \cite{beyersmann13} {\it time-simultaneous confidence bands} can be constructed by the functional delta method 
as follows:
\begin{enumerate}
 \item We consider the transformed Aalen-Johansen estimator $\gamma_n(t) = \sqrt{n} g(t) \{ \phi(\widehat F_1(t)) - \phi(F_1(t)) \}$
 \item with transformation $\phi$ (such as $\phi_1(t) = \log ( - \log(1-t))$),
 \item weight function $g$ (such as $g_1(t) = \log(1-\widehat F_1(t))/\widehat \sigma(t)$ or $g_2(t) = \log(1-\widehat F_1(t))/ (1+ \widehat \sigma^2(t))$),
 \item variance estimator $\widehat \sigma^2(t) = n \widehat \var (\widehat F_1(t)) / (1-\widehat F_1(t))^2$,
 \item and its corresponding resampling version $\widehat \gamma_n(t) = g(t) \phi'(\widehat F_1(t)) \widehat W_n^D$.
\end{enumerate}
The variance $n \widehat \var (\widehat F_1(t))$ in the DDMB resampling version $\widehat \gamma_n$ 
is similar to the wild bootstrap variance estimator of~\cite{dobler14a}, where we now use the same DDMB weights as in $\widehat W_n^D$. 
Again following \cite{lin97} and \cite{beyersmann13} we call the bands resulting from $g_1$ and $g_2$ equal precision and Hall-Wellner bands, respectively.
Simulating the 95\% quantile $q_{.95}$ of $\sup_{t \in I} | \widehat \gamma_n(t) |$
(thereby keeping the data fixed) and using the transformation $\phi_1$,
approximate 95\% confidence bands for $(F_1(t))_{t \in I}$ are obtained as
\begin{align}
\label{eq:cb}
  1 - (1 - \widehat F_1(t))^{\exp(\pm q_{.95} / (\sqrt n g(t)) )}, \quad t \in I.
\end{align}
Equivalent {\it Kolmogorov-Smirnov-type tests} $\varphi$ for the null hypothesis $H_= : \{ F_1 = F \text{ on } I \}$ for a prespecified function $F$ 
are given as $\varphi = 0$ if and only if $F$ is completely contained in the above confidence band.

{\color{black}
Finally, {\it pointwise confidence intervals} for the binomial probability $F_1(s) = P(X_s = 1)$ for each $s \in I$ are immediately obtained by letting
$t_1 = t_2 = s$ so that $I = \{s\}$.}

\subsection{Two-Sample Resampling Tests for Equal CIFs}

%
%
Another topic of interest is the comparison of two CIFs for the same risk but from independent sample groups with sample sizes $n_1$ and $n_2$, respectively.
For this reason we introduce all quantities of the previous sections sample-specifically
and denote them with a superscript $^{(k)}$, $k=1,2$. 
For example, $F_1^{(2)}$ is the second group's CIF for the first risk,
$\tau^{(1)}$ is the terminal point for observations in the first group
and $D_{2n;i}^{(k)}$ is the DDMB weight for $Z_{2n;i}^{(k)}$, where $n = n_1 + n_2$.
Further, we define $\tau = \tau^{(1)} \wedge \tau^{(2)}$ and $\mac{A}_n = \sigma(\mac{A}_{n_1}^{(1)}, \mac{A}_{n_2}^{(2)})$. 
We would now like to construct non-parametric resampling tests for the hypotheses
\begin{align*}
  H_= : \{F_1^{(1)} = F_1^{(2)} \text{ on } [t_1,t_2] \} 
    \text{ versus } K_{\neq} : \{F_1^{(1)} \neq F_1^{(2)} \text{ on a subset } A \subseteq [t_1,t_2] \text{ such that } \lambda\!\!\lambda(A) > 0 \},
\end{align*}
where $\lambda\!\!\lambda$ denotes Lebesgue measure.
To this end we first introduce the two-sample version of~\eqref{eq:w} as a scaled difference of Aalen-Johansen estimators, namely
\begin{align*}
 W_{n_1,n_2} = \sqrt{\frac{n_1 n_2}{n}} ( \widehat F_1^{(1)} - \widehat F_1^{(2)} ) 
  = \sqrt{\frac{n_2}{n}} W_{n_1}^{(1)} - \sqrt{\frac{n_1}{n}} W_{n_2}^{(2)} + \sqrt{\frac{n_1 n_2}{n}} ( F_1^{(1)} - F_1^{(2)} ).
\end{align*}
Based on a similar martingale representation as in Equation~\eqref{eq:martingale rep},
we arrive at a DDMB version of $W_{n_1,n_2}$,
\begin{align}
\label{eq:w_d_hut}
 \widehat W^D_{n_1,n_2} = \sqrt{\frac{n_1 n_2}{n}} \Big( \sum_{i=1}^{2n_1} D_{2n;i}^{(1)} Z_{2n;i}^{(1)} + \sum_{i=1}^{2n_2} D_{2n;i}^{(2)} Z_{2n;i}^{(2)} \Big),
\end{align}
see \eqref{eq:wildbstat b} for the corresponding one-sample case. 
This gives us a generalization of the two-sample wild bootstrap statistic of~\cite{dobler14b}
where such resampling tests based on i.i.d. multipliers are compared to computationally less expensive approximate tests.

Following the lines of~\cite{dobler14b}, we now construct several resampling tests for $H_=$ versus $K_{\neq}$.
This is accomplished by plugging the statistic $W_{n_1,n_2}$ and its resampled version $\widehat W^D_{n_1,n_2}$ 
into a continuous functional $\psi: \D[0,\tau] \rightarrow \R$
such that $\psi(W_{n_1,n_2})$ tends to infinity in probability if the alternative hypothesis $K_{\neq}$ is true. 
In this subsection the asymptotic statements are referred to as $n \rightarrow \infty$ and $n_1/n \rightarrow \kappa\in(0,1)$.  
Since $W_{n_1,n_2}$ and $\widehat W^D_{n_1,n_2}$ possess the same Gaussian limit distribution,
the resulting test depending on $\psi(W_{n_1,n_2})$ (as test statistic) 
and $\psi(\widehat W^D_{n_1,n_2})$ (yielding a data-dependent critical value) is of asymptotic level $\alpha$.
Furthermore, the test is consistent, that is, it rejects the alternative hypothesis $K_{\neq}$ with probability tending to 1 as $n \rightarrow \infty$.
Thus, the following two theorems follow immediately from the weak convergence results 
of the preceding theorem for $W_{n_1,n_2}$ and $\widehat W^D_{n_1,n_2}$ 
and from applications of the continuous mapping theorem.
\begin{thm}[A Kolmogorov-Smirnov-type test]
 \label{thm:resampling_wild_cif_equality_2}
 Choose a triangular array of DDMB weights $D_{2n;i}^{(k)}, i=1,\dots,2 n_k, \\ k=1,2,$ 
 satisfying~\eqref{eq:conv_wbs_1} -- \eqref{eq:conv_wbs_6}
 and let $w: [t_1,t_2] \rightarrow (0,\infty)$ be a bounded weight function.
 A consistent, asymptotic level $\alpha$ resampling test for $H_=$ vs. $K_{\neq}$ is given by
  \begin{align*}
    \varphi^{KS} = \left\{
    \begin{array}{rlc}
      1 && > c^{KS} \\
	&	\sup_{u \in [t_1,t_2]} w(u) | W_{n_1,n_2}(u) |	&  \\
      0	&& \leq c^{KS}
    \end{array}
    \right.
  \end{align*}
  where $c^{KS}( \cdot )$ is the $(1-\alpha)$-quantile of the conditional distribution 
  \begin{align*}
  \mac{L} \Big( \sup_{u \in [t_1,t_2]} w(u) | \widehat W_{n_1,n_2}^D (u) | \; \Big| \; \mac{A}_n \Big).
  \end{align*}
\end{thm}
\begin{thm}[A Cram\'{e}r-von Mises-type test]
 \label{thm:resampling_wild_cif_equality_1}
 Choose a triangular array of DDMB weights $D_{2n;i}^{(k)}, i=1,\dots,2 n_k, \\ k=1,2,$ 
 satisfying~\eqref{eq:conv_wbs_1} -- \eqref{eq:conv_wbs_6}
 and let $w: [t_1,t_2] \rightarrow (0,\infty)$ be an integrable weight function.
 A consistent, asymptotic level $\alpha$ resampling test for $H_=$ vs. $K_{\neq}$ is given by
  \begin{align*}
    \varphi^{CvM} = \left\{
    \begin{array}{rlc}
      1 && > c^{CvM} \\
	&	\int_{t_1}^{t_2} w(u) W^2_{n_1,n_2}(u) \d u	&  \\
      0	&& \leq c^{CvM}
    \end{array}
    \right.
  \end{align*}
  where $c^{CvM}( \cdot )$ is the $(1-\alpha)$-quantile of the conditional distribution 
  \begin{align*}
   \mac{L}\Big(  \int_{t_1}^{t_2} w(u) ( \widehat W_{n_1,n_2}^D )^2 (u) \d u 
    \; \Big| \; \mac{A}_n \Big).
  \end{align*}
\end{thm}
\begin{rem}
 For given $(D_{2n;i}^{(k)})_{i,k}$ we could also choose the DDMB weights as the slightly modified variables 
 $\tilde D_{2n;i}^{(k)} = ( 1 + o_p(1)) D_{2n;i}^{(k)}$ for asymptotically negligible terms $o_p(1)$ 
 which are supposed to be measurable w.r.t. $\mac{A}_n$.
 In the article of~\cite{dobler14a} it is seen that wild bootstrap tests may tend to be slightly too liberal
 for strongly unequal sample sizes or when censoring is present.
 Therefore, the choice of, for instance, $o_p(1) = o(1) = \frac{| n_1 - n_2 |}{n_1 n_2}$ or its square root 
 leads to slightly more conservative versions of the above tests in case of unequal sample sizes.
 {\color{black}
 In order to additionally account for censoring, we could even choose the rather bigger $o_p(1) = \frac{| Y^{(1)} - Y^{(2)} |}{Y^{(1)} Y^{(2)}}(t_2)$ 
 (assuming approximately equal censoring rates in both groups)
 since the denominator tends to be smaller the more individuals are censored.}
\end{rem}

%
%
%
%
%
%
%
%
%
%
%
%
%
%
%
%
%
%
%
\section{Simulations}
\label{sec: simus} {\color{black}The aim of the present simulation study is to
  assess the coverage probabilities of confidence bands for the first CIF in a situation similar
  to the real data example which is introduced and analyzed in
  Section~\ref{sec: example}. To this end, ties in the original data set have
  been broken. Data have been simulated from smoothed versions of the
  non-parametric estimators; see \cite{bmc2011} for a similar
  approach. Table~\ref{tab:event_perc} reports comparable percentages of type
  1 and type 2 events as well as of censorings for both the original data set
  and 50,000 simulated individuals.}



\begin{table}[ht]
\centering
\begin{tabular}{c|cc}
  \hline
  \hline
& data-set & simulations \\ \hline
  type 1 events & 38.21 & 38.68 \\ 
  type 2 events & 20.28 & 20.06 \\ 
  censorings & 41.51 & 41.26 \\ 
   \hline
\end{tabular}
\caption{Percentages of types of observations.} 
\label{tab:event_perc}
\end{table}

The simulations were conducted using the R-computing environment, version
3.1.3 (R Development Core Team, 2015), each with $N_{sim} = 10,000$ simulation
runs for simulations with up to 100 individuals under study.  For larger
groups of individuals, we have chosen $\tilde N_{sim} = 1000$ simulation runs
due to the enormously increasing computational efforts.  For determination of
the random quantile $q_{.95}$ we have run $B= 999$ bootstrap runs in each
simulation step.  We {\color{black}constructed} both Hall-Wellner and equal
precision bands {\color{black}on the time interval $[.5, 5]$}, each based on
either standard normal, centered $Poi(1)$ or weird bootstrap weights within the
DDMB approach.  Table~\ref{tab:cps} gives the resulting coverage probability
estimates for $n \in \{ 50, 60, \dots, 100, 200, 300, 636 \}$ simulated
individuals under study in each simulation run, {\color{black}where $n=636$ is
  the sample size of the data example studied in Section~\ref{sec: example}.}

{\color{black}All coverage probabilities in Table~\ref{tab:cps} are too small
  for sample sizes~$n\le 200$, but with a tendency of better coverage
  probabilities for Poisson multipliers and the
  weird bootstrap. For these two resampling procedures, there is also a
  preference for equal precision bands. This is also the scenario which
  draws near to the nominal level for~$n=300$, while standard normal multipliers
  lead to a coverage probability less than~$91\%$ for both types of bands. All
  bootstrap variants approach the nominal level for~$n=636$. 
  Finally, Figure~\ref{fig:cbs_4D} in the subsequent section shows an
  empirical probability of~$51/636\approx 8.0\%$ for being at risk
  at~$t=5-$ which reinforces the impression that the construction of bands was
  an ambitious aim for sample sizes of~$n\le 100$. }



\begin{table}[ht]
        \centering
        \begin{subtable}{0.45\textwidth}
	  \begin{tabular}{c|ccc}
	    \hline
	    \hline
	    n & normal & Poisson & weird \\ \hline
	    50 & 79.4 & 80.77 & 79.84 \\ 
	    60 & 82.45 & 82.68 & 82.67 \\ 
	    70 & 84.86 & 85.59 & 85.44 \\ 
	    80 & 86.2 & 86.74 & 86.86 \\ 
	    90 & 87.9 & 88.21 & 88.49 \\ 
	    100 & 88.22 & 89.07 & 89.50 \\ 
	    200 & 89.9 & 91.5 & 92.1 \\
	    300 & 90.9 & 93.6 & 93.1 \\
	    636 & 94.8 & 94.1 & 94.9 \\\hline
	  \end{tabular}
	  \caption{Hall-Wellner bands} 
	\end{subtable}
	\begin{subtable}{0.45\textwidth}
	  \begin{tabular}{c|ccc}
	    \hline
	    \hline
	    n & normal & Poisson & weird \\ \hline
	    50 & 76.49 & 80.11 & 79.72 \\ 
	    60 & 80.44 & 84.22 & 83.43  \\ 
	    70 & 82.89 & 86.34 & 86.49 \\ 
	    80 & 85.36 & 87.93 & 88.22 \\ 
	    90 & 86.05 & 89.67 & 89.38 \\ 
	    100 & 87.68 & 90.55 & 91.06 \\ 
	    200 & 91.1 & 93.3 & 93.9 \\
	    300 & 90.6 & 95.2 & 95.1 \\
	    636 & 94.1 & 95.7 & 94.4 \\\hline
	  \end{tabular}
	  \caption{Equal precision bands} 
	\end{subtable}
        \caption{Per cent coverage probabilities of confidence bands.}
	\label{tab:cps}
\end{table}

\newpage

\section{A Real Data Example}
\label{sec: example}

{\color{black} We consider data from the 4D study \citep{wanner05}, which was a
  prospective randomized controlled trial evaluating the effect of lipid
  lowering with atorvastatin in diabetic patients receiving hemodialysis. The
  primary outcome was a composite of death from cardiac causes, stroke, and
  non-fatal myocardial infarction, subject to the competing risk of death from
  other causes. The motivation of the trial was that statins are protective
  with respect to cardiovascular events for persons with type 2 diabetes
  mellitus without kidney disease, but a possible benefit in patients
  receiving hemodialysis had until then not been assessed. \cite{g05:_sampl}
  have discussed sample size planning with competing risks outcomes for the 4D
  study, and \cite{bmc2011} have used the 4D study to advocate a simulation
  point of view for the interpretation of competing risks.

  \cite{wanner05} found a non-significant protective effect of atorvastatin on
  the cause-specific hazard of the primary outcome (hazard ratio $0.92$ with
  95\%-confidence interval [0.77, 1.10]). There was essentially no difference
  between groups for the competing cause-specific hazard, implying similar
  CIFs in the groups; see \cite{bmc2011} for an in-depth discussion. Hence, we
  restrict ourselves in this section to a one sample scenario and re-analyze
  the control group data ($636$~patients). The data have been made available
  in the \verb=R=-package \emph{etm} \citep{beyersmann12}, and our results may
  therefore be checked for reproducibility. Ties have been broken as in
  Section~\ref{sec: simus}.} 


{\color{black}Figure~\ref{fig:cbs_4D} shows Hall-Wellner (left panel) and equal
  precision (right panel) bands for the CIF of the primary outcome, using the
  weird bootstrap and the wild bootstrap with both, standard normal and centered $Poi(1)$ weights.
  Within each panel, differences between the bands are
  invisible to the naked eye. Table~\ref{tab:areas} additionally shows the
  areas between upper and lower boundary of the confidence bands; differences
  are again negligible.}


The only notable difference is the form of both types of bands: While the
Hall-Wellner bands' boundaries seem to have almost the same distances for all
points of time, the equal precision bands start with a narrower band at $t =
.5$ which clearly becomes wider as time progresses.  But eventually, the areas
of both types of bands are again comparable.

{\color{black}Figure~\ref{fig:cbs_4D} additionally shows pointwise,
  log-log-transformed confidence intervals. As expected, the pointwise
  intervals are narrower than the simultaneous bands, but the bands do perform
  competitively.}
  
\begin{figure}[ht]
        \centering
        \begin{subfigure}{0.5\textwidth}
                \includegraphics[width=\textwidth, height=0.8\textwidth]{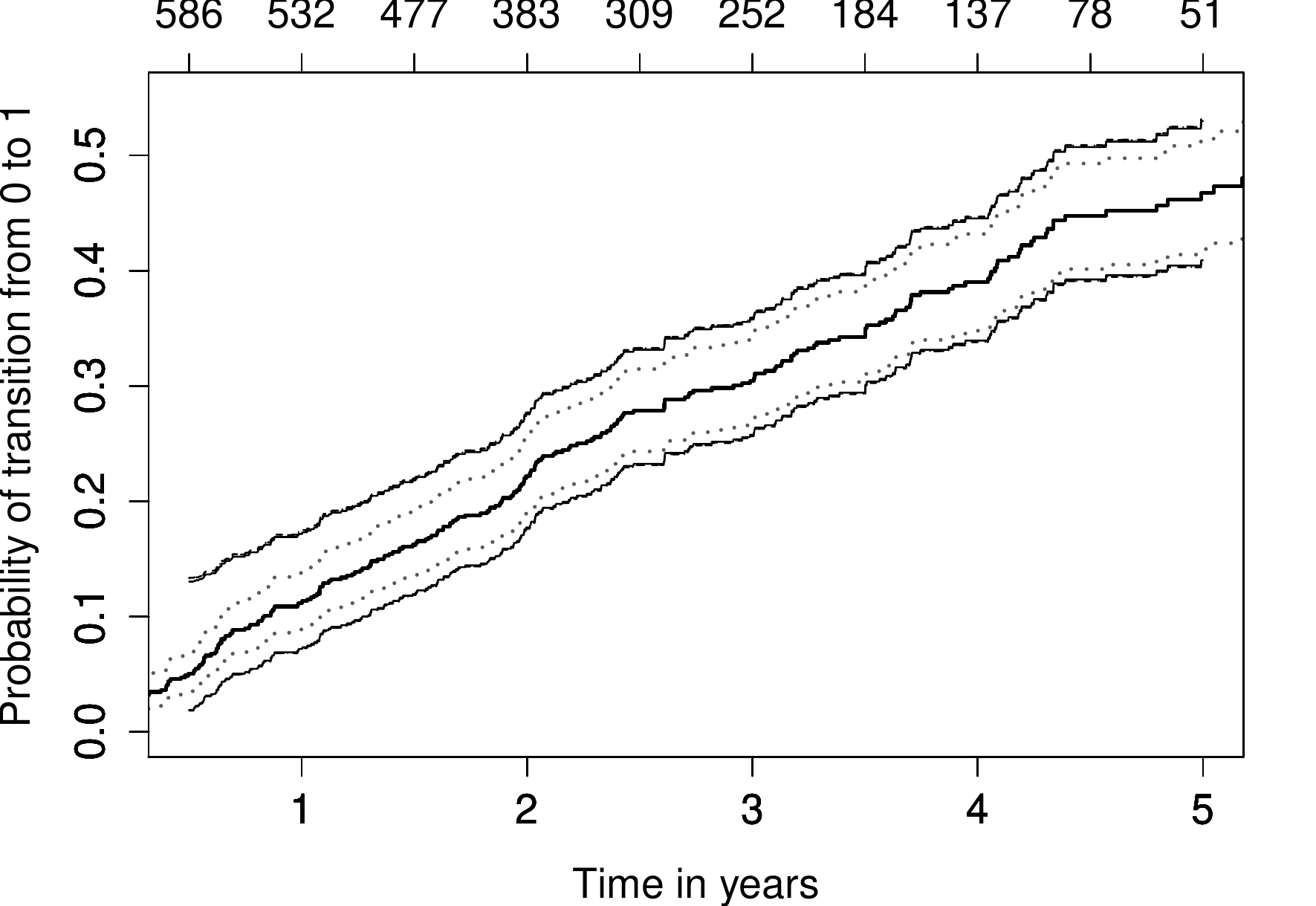}
                \caption{Hall-Wellner bands for $n=636$}
        \end{subfigure}%
        \begin{subfigure}{0.5\textwidth}
                \includegraphics[width=\textwidth, height=0.8\textwidth]{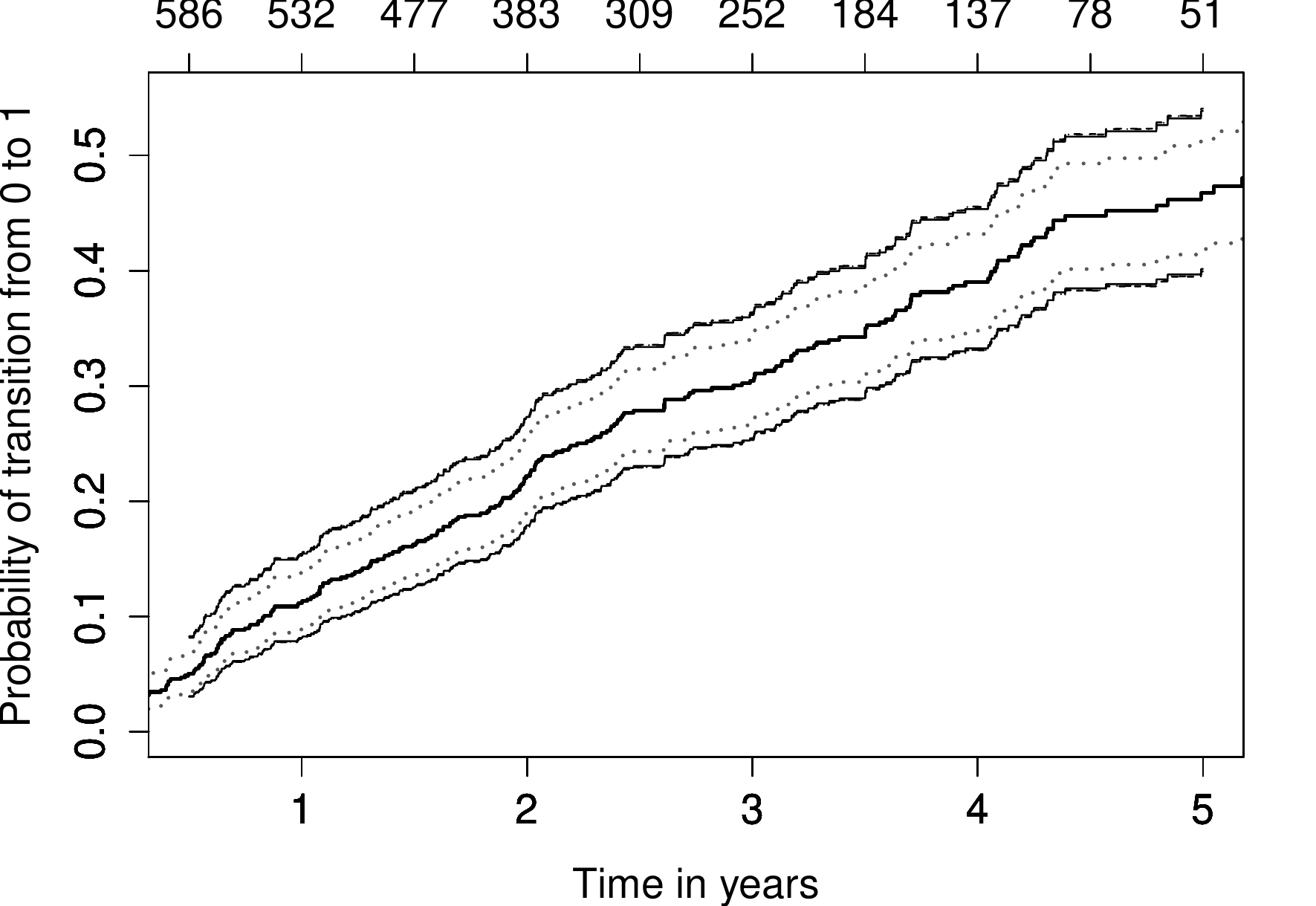}
                \caption{Equal precision bands for $n=636$}
        \end{subfigure}
        \caption{Approximate 95\% confidence bands for $F_1$ using different DDMB weights:
	  standard normal (----), \\ centered Poi(1) (- - -), weird bootstrap ($\cdot - \cdot$) multipliers.
	  The solid line in the middle is the corresponding Aalen-Johansen estimator.
	  Pointwise 95\% confidence intervals ($\cdots\cdot$) for $F_1$ also based on a $\log-\log$ transformation, plotted in dark grey,
	  have been calculated using the \texttt{R}-package \emph{etm}.
	  Above of the plots the number of individuals under risk shortly before each half-year is indicated.}
      \label{fig:cbs_4D}
\end{figure}

\begin{table}[ht]
\centering
	\begin{tabular}{c|cc}
	  \hline
	  \hline
	   & Hall-Wellner & Equal precision \\ \hline
	  normal & .4655 & .4621 \\ 
	  Poisson & .4783 & .4770 \\ 
	  weird & .4764 & .4746 \\ 
	  \hline
	\end{tabular}
	\caption{Areas covered by the confidence bands in Figure~\ref{fig:cbs_4D} for different resampling schemes.} 
	\label{tab:areas}
\end{table}

We also performed analogous analyses in a data subsample with 200 and 300 individuals. 
In line with our simulation results, 
the wild bootstrap with standard normal multipliers produced narrower bands, 
but - similar to the complete cohort - the differences between the different bands were of little practical importance in this example. 
In the analyses of the subsample, the bands again performed competetively when compared to pointwise confidence intervals. 
(Results not shown.)


\section{Discussion and Outlook}
\label{sec: dis} {\color{black}We have introduced and rigorously justified the
  new data-dependent multiplier bootstrap for non-parametric analysis of
  survival data. Observation may be restricted by independent right-censoring
  and left-truncation, but a strict i.i.d.\ setup is not required. Our
  developments have included the case where failure may be due to several
  competing risks, where resampling is particularly attractive due to lack of
  asymptotic pivotal approximations. Our general framework includes both the wild
  bootstrap and the weird bootstrap as special cases. The wild bootstrap with
  standard normal multipliers is a popular and computationally convenient
  technique \citep[e.g.,][]{Mart:Sche:dyna:2006}. The weird bootstrap,
  introduced by \cite{abgk93} in their essential book on \emph{Statistical
    Models Based on Counting Processes}, appears to be rarely used, if at all,
  although it has been implemented in software. To the best of our knowledge,
  our paper is the first to rigorously show asymptotic correctness of the
  weird bootstrap in the present context.  The variety of available resampling
  techniques raises the question of which bootstrap to use. Efron's original
  proposal of repeatedly taking random samples with replacement from the
  randomly censored observations \citep{efron81} is arguably closest to his
  original approach \citep{efron79}, but does rely on a strict i.i.d.\ setup;
  see also the discussion in \citet[Section~IV.1.4]{abgk93}.  The wild
  bootstrap with standard normal multipliers is motivated by the martingale
  representations used in the proofs of weak convergence of the original
  estimators. In a nutshell, the idea is to replace asymptotic normality by
  finite sample normality (because of normal multipliers, keeping the data
  fixed) with approximately the right covariance. The general wild bootstrap
  allows for non-normal multipliers, replacing finite sample normality by
  approximate normality.  But the weird bootstrap is perhaps the most natural
  resampling scheme for survival data. To see this, recall that one major
  reason for basing survival analysis on hazards is censoring. In our setting,
  and assuming for the time being independent random censorship by $C$, we have
  that \begin{equation*} \alpha_{j}(t) {\rm d}t = P(T\in [t, t + {\rm d}t),
    X_T=j\,|\, T\ge t)=P(T\in [t, t + {\rm d}t), X_T=j, T\le C\,|\, T\ge t,
    C\ge t), \end{equation*} where the first equality is the definition 
    from Equation \eqref{eq:alpha} and the second equality follows because of random
  censoring. Independent censoring now essentially requires the last equation
  (reformulated using counting processes and at-risk processes) to hold rather
  than the existence of a latent censoring time, which is assumed to be
  stochastically independent of~$(T,X_T)$.  It is the second equality that,
  first of all, motivates the increments of the cause-specific NAE, say ${\rm
    d}\widehat{A}_j(t)={\rm d}N_j(t)/Y(t)$. The weird bootstrap continues from
  this point by sampling $B(Y(t),{\rm d}\widehat{A}_j(t))$-distributed
  increments at the jump times of
  $N$.  The fact that sampling is performed independently at the jump times
   is justified by the asymptotic distribution of
  $\sqrt{n}(\widehat{A}_j-A_j)$
   having independent increments.

  Our simulation results have shown that one should keep alternatives to the
  wild bootstrap with the almost exclusively used standard normal multipliers
  in mind. In the scenarios that we have considered, we found a preference for
  Poisson multipliers and for the weird  bootstrap. \cite{beyersmann13} who only considered the wild bootstrap also
  found a preference for Poisson multipliers, but the differences in the
  present paper were more pronounced. We did not find noticeable differences
  between the approaches in the real data example, but our analysis
  illustrated that simultaneous confidence bands may perform competitively
  when compared to only pointwise confidence intervals. Such bands should be
  reported more often, because subject matter interest often does lie in
  survival \emph{curves} rather than probabilities at fixed time
  points.

  We are currently investigating extensions of the new DDMB approach to
  multi-state and regression models, see e.g. \cite{lin2000semiparametric} or \cite{scheike2003extensions} 
  for a normal multiplier application. In particular, the weird bootstrap naturally extends to these situations.  }


\section*{Acknowledgements}

The authors like to thank Arthur Allignol and Arnold Janssen for helpful
discussions and Marc Ditzhaus for computational help.  Moreover, the authors
Dennis Dobler and Markus Pauly appreciate the support received by the SFF
grant F-2012/375-12. {\color{black}Jan Beyersmann was supported by Grant BE
  4500/1-1 of the German Research Foundation (DFG).}

\section{Appendix}
\label{sec:appendix}
The conditional convergence of the finite-dimensional marginal distributions of a linear, resampled process statistic with DDMB weights
can be concluded with the following lemma
which generalizes Theorem~A.1 in~\cite{beyersmann13}.
To this end let $\| \cdot \|$ be a norm on $\R^d$, $d \in \N$, and define $\widehat {\b{S}}_n^D = \sum_{i=1}^n D_{n;i} \bs \xi_{n;i}$.
In the following let $\mac{C}_n$ be a $\sigma$-field which contains $\sigma(\bs \xi_{n;i}: i = 1, \dots, n)$.
$\bs \xi_{n;i}$ and $D_{n;i}$ are specified in the following lemma.
\begin{lemma}
\label{lemma:fidi_conv_general_wild_bs}
  Let the triangular array of random variables
  $(D_{n;i})_{i = 1,\dots, n}: \Omega \rightarrow \R^n$ 
  with finite second moments
  and the triangular array of $\R^d$-valued random vectors $ (\bs \xi_{n;i})_{i = 1,\dots, n}$
  fulfill the following six conditions:
  \begin{align}
   & \sum_{i=1}^n \bs \xi_{n;i} \bs \xi_{n;i}^T \oPo \Gamma, \quad \text{where $\Gamma$ is a positive definite covariance matrix,}  \label{eq:wild_bs_1} \\
   & \max_{i=1, \dots, n} \| \bs \xi_{n;i} \| \oPo 0, 
      \label{eq:wild_bs_2} \\
   & \sqrt{n} \max_{i=1,\dots,n} | \mu_{n;i} | := \sqrt{n} \max_{i = 1,\dots,n} | \E[ D_{n;i} | \mac{C}_n] | \oPo 0, \label{eq:wild_bs_3} \\
   & \max_{i=1,\dots,n} | \sigma_{n;i}^2 - 1 | := \max_{i = 1,\dots,n} | \var(D_{n;i} | \mac{C}_n ) - 1 | \oPo 0
      \text{ as } n \rightarrow \infty, \label{eq:wild_bs_4} \\
   & \mac{L} ( D_{n;i}, i = 1,\dots,n  \; | \; \mac{C}_n )
   = \bigotimes_{i = 1,\dots,n} \mac{L} ( D_{n;i} \; | \; \mac{C}_n ).
      \label{eq:wild_bs_5}
   \end{align}
   In addition, the weights $(D_{n;i})_{i = 1,\dots, n}$ may satisfy the Lindeberg condition in probability given $\mac{C}_n$, that is
   \begin{align}
   \label{eq:wild_bs_6}
    \sum_{i=1}^n \E \Big[ \frac{(D_{n;i} - \mu_{n;i} )^2}{\sum_{j=1}^n \sigma_{n;j}^2} 
      \textnormal{\b 1} \Big( \frac{(D_{n;i} - \mu_{n;i})^2}{\sum_{j=1}^n \sigma_{n;j}^2} > \varepsilon \Big) \Big| \mac{C}_n \Big] \oPo 0
      \qquad\text{ for all } \varepsilon > 0.
  \end{align}
  Then the conditional weak convergence $\widehat {\b S}_n^D \oDo N(0, \Gamma)$
  given $\mac{C}_n$ holds in probability.
\end{lemma}
{\it Proof.} 
  Following the proof in~\cite{beyersmann13}
  we only need to show that $\widehat { S}_n^D$ satisfies the conditional Lindeberg condition for dimension $d=1$.
  The case for general $d \in \N$ follows from a modified Cram\'er-Wold Theorem; see~\cite[Theorem 4.1]{pauly11} for details.
  Thus, we calculate
  \begin{align}
  \label{eq:helpful2}
    \Gamma^D_n := \var(\widehat { S}_n^D | \mac{C}_n) = \sum_{i=1}^n \var(D_{n;i} \xi_{n;i} | \mac{C}_n )
    = \sum_{i=1}^n \xi_{n;i}^2 \sigma_{n;i}^2 \oPo \Gamma > 0
  \end{align}
  by~\eqref{eq:wild_bs_1} and \eqref{eq:wild_bs_4}.
  Further, we write $\widehat S_n^D = \sum_{i=1}^n (D_{n;i} - \mu_{n;i}) \xi_{n;i} + \sum_{i=1}^n \mu_{n;i} \xi_{n;i} =: S_n^D + R_n^D$
  of which $R_n^D$ is asymptotically negligible by Cauchy-Schwarz' inequality, Conditions~\eqref{eq:wild_bs_1} and \eqref{eq:wild_bs_3}
  and Slutzky's theorem:
  \begin{align*}
   | R_n^D| = \Big| \sum_{i=1}^n \mu_{n;i} \xi_{n;i} \Big| \leq \max_{i=1,\dots,n} | \mu_{n;i} | \sum_{i=1}^n | \xi_{n;i} | 
    \leq \max_{i=1,\dots,n} | \mu_{n;i} | \sqrt{n} \Big( \sum_{i=1}^n  \xi_{n;i}^2  \Big)^{1/2} \oPo 0 \cdot \sqrt \Gamma.
  \end{align*}
  It remains to verify the conditional Lindeberg condition for $S_n^D$ in probability
  where we let $\E[ D_{n;i} | \mac{C}_n]  = 0$ without loss of generality.
  For this last step we need that $\sum_{i=1}^n ( \xi^2_{n;i} - \Gamma ) \sigma_{n;i}^2 \oPo 0$
  which can be easily shown using Condition~\eqref{eq:wild_bs_4} and the convergence in~\eqref{eq:helpful2}.
  Thus, it follows that for all $\delta > 0$,
  \begin{align*}
   P \Big( - \delta + \Gamma \sum_{j=1}^n \sigma_{n;j}^2 \leq \Gamma^D_n \leq \delta + \Gamma \sum_{j=1}^n \sigma_{n;j}^2 \Big) \oPo 1.
  \end{align*}
  Now, for all $\delta, \varepsilon, \eta > 0$ sufficiently small and $n$ sufficiently large we have
  \begin{align*}
   & (\Gamma^D_n)^{-1} \sum_{i=1}^n \E \left[\left. D_{n;i}^2 \xi_{n;i}^2 \b 1 ( (D_{n;i} \xi_{n;i} )^2 > \Gamma^D_n \varepsilon) \right| 
      \mac{C}_n\right] \\
   & \leq \Big( \Gamma \sum_{j=1}^n \sigma_{n;j}^2 - \delta \Big)^{-1} \max_{i=1, \dots,n} \xi_{n;i}^2
      \sum_{i=1}^n \E \Big[ D_{n;i}^2 
	  \b 1 \Big( (D_{n;i} \xi_{n;i} )^2 > \Big( \Gamma \sum_{j=1}^n \sigma_{n;j}^2 - \delta \Big) \varepsilon \Big) \; \Big| \; \mac{C}_n \Big] + o_p(1) \\
   & \leq \Big( \frac{\Gamma}{2} \sum_{j=1}^n \sigma_{n;j}^2 \Big)^{-1} \b 1 \big( \max_{i=1, \dots, n} | \xi_{n;i} | < \eta \big)
   \eta^2 \sum_{i=1}^n \E \Big[ D_{n;i}^2 
      \b 1 \Big( \eta^2 D_{n;i}^2 > \Big( \frac{\Gamma}{2} \sum_{j=1}^n \sigma_{n;j}^2 \Big) \varepsilon \Big) \; \Big| \; \mac{C}_n \Big] 
	 + o_p(1)
  \end{align*}
  which is $o_p(1)$ 
  by \eqref{eq:wild_bs_6}.
  Therefore, $\widehat S_n^D$ satisfies the Lindeberg condition given $\mac{C}_n$ in probability. \qed
\begin{rem}
\label{rem:apply_lemmas}
  See \cite{beyersmann13} to
  note that Conditions~\eqref{eq:wild_bs_1} and~\eqref{eq:wild_bs_2} are fulfilled for 
  the triangular array $\bs \xi_n = \sqrt{\frac{n_1 n_2}{n}} (\b Z_{2n}^{(k)}(t_j))_{j,k}$ and $\mac{C}_n = \mac{A}_n$ 
  where, for each $j=1, \dots, \ell, k=1,2,$ the vector $\b Z_{2n}^{(k)}(t_j)=(Z_{2n;1}^{(k)}(t_j), \dots, Z_{2n;2n_k}^{(k)}(t_j))$ consists of the integrals w.r.t. counting processes given by~\eqref{eq:wildbstat} and~\eqref{eq:w_d_hut} evaluated at arbitrary times $t_1, \dots, t_\ell \in I$.
  Moreover, this choice for $\bs \xi_n$ also fulfills the conditions of Lemma~\ref{lemma:tightness_wbs} below.
\end{rem}
Let us now give a criterion for the tightness of linear, resampled process statistics
in terms of the DDMB weights $(D_{n;i})_{i=1,\dots,n}$
and the data vectors $(\bs \xi_{n;i})_{i=1,\dots,n}$.
Since tightness of a family of multivariate processes is equivalent to the tightness in each dimension,
we here only consider the case of $d=1$.
Recall the $\mac{O}_p$-notation introduced above Theorem~\ref{thm:ddmb_cifs}.
\begin{lemma}
\label{lemma:tightness_wbs}
 Let each $\xi_{n;i}: \Omega \times I \rightarrow \mathbb{R}, i=1,\dots,n,$ be a stochastic process and
 suppose that, as $n \rightarrow \infty$,
  \begin{align}
   \label{eq:tight_wbs_1} & \max_{i=1,\dots,n} | \E[ D_{n;i} | \mac{C}_n] | \in \mac{O}_p(n^{-1/2}), \\
   \label{eq:tight_wbs_2} & \max_{i=1,\dots,n} \E[ D^2_{n;i} | \mac{C}_n] \in \mac{O}_p(1), \\
   \label{eq:tight_wbs_3} & \max_{i=1,\dots,n} | \E[ D^3_{n;i} | \mac{C}_n] | \in \mac{O}_p(\sqrt{n/r_n}), \\
   \label{eq:tight_wbs_4} & \max_{i=1,\dots,n} \E[ D^4_{n;i} | \mac{C}_n] \in \mac{O}_p(r_n^{-1}), \\
   \label{eq:tight_wbs_5} & \mac{L}( D_{n;i}, i = 1,\dots,n \; | \; \mac{C}_n )
    = \bigotimes_{i = 1,\dots,n} \mac{L}( D_{n;i} \; | \; \mac{C}_n ),  \\
   \label{eq:tight_wbs_7} & \sum_{i=1}^n (\xi_{n;i}(s) - \xi_{n;i}(r))^2 \leq H_n(s) - H_n(r) \oPo H(s) - H(r), \quad 0 \leq r \leq s, \; r,s \in I,
  \end{align}
  where $H, H_n: \Omega \times I \rightarrow [0,\infty)$ are nondecreasing functions of which $H$ is continuous and deterministic
  and where $r_n = \max_{i=1,\dots,n} \max_{s,t \in I} \frac{(\xi_{n;i}(t) - \xi_{n;i}(s))^2}{\sum_{j=1}^n (\xi_{n;j}(t) - \xi_{n;j}(s))^2} \in (0,1)$.
  Then the family of probability measures $\mac{L} \big( \widehat S_n^D | \mac{C}_n \big)$ is tight in probability.
\end{lemma}
\begin{proof}
  By its analogy to the proof of tightness for the exchangeably weighted bootstrapped Aalen-Johansen process in~\cite{dobler14a},
  where the moment conditions for the (mixed) moments are now replaced by~\eqref{eq:tight_wbs_1} -- \eqref{eq:tight_wbs_5},
  we only need to consider the asymptotics of the involved moments therein; see the proof of their Theorem~3.1.
  In fact, moving on to the conditional expectations essentially does not effect the arguments of the referred proof.
  It is sufficient to verify that the existing proof holds with these modifications.
  
  Note that we here analyze the conditional moments of $\widehat S_n^D$ 
  without previously centering the DDMB weights at their arithmetic mean
  which had been necessary in the article by \cite{dobler14a}.
  
  Two of those five cases emerging in the referred proof require a separate consideration
  since our Lemma~\ref{lemma:tightness_wbs} is formulated in a greater generality.
  Therefore, we begin to note that, in the first sum on the right-hand side of~(A.3) in~\cite{dobler14a}, where $\E[ D^4_{n;i} | \mac{C}_n]$ occurs, 
  we also have factors like
  \begin{align*}
   & \sum_{i=1}^n (\xi_{n;i}(t) - \xi_{n;i}(s))^2 (\xi_{n;i}(s) - \xi_{n;i}(r))^2 \\
   & = \sum_{i=1}^n \frac{(\xi_{n;i}(t) - \xi_{n;i}(s))^2}{\sum_{j=1}^n (\xi_{n;j}(t) - \xi_{n;j}(s))^2}  
      (\xi_{n;i}(s) - \xi_{n;i}(r))^2 \sum_{j=1}^n (\xi_{n;j}(t) - \xi_{n;j}(s))^2 \\
    & \leq r_n \sum_{j=1}^n (\xi_{n;j}(t) - \xi_{n;j}(s))^2 \sum_{i=1}^n (\xi_{n;i}(s) - \xi_{n;i}(r))^2.
  \end{align*}
  This is why~\eqref{eq:tight_wbs_4} is sufficient for having reasonable upper bounds of this first sum.
  A similar argument is required for those sums where third moments occur, i.e.,
  \begin{align*}
   & \sum_{i=1}^n (\xi_{n;i}(t) - \xi_{n;i}(s)) (\xi_{n;i}(s) - \xi_{n;i}(r))^2 \\
   & = \sum_{i=1}^n \frac{(\xi_{n;i}(t) - \xi_{n;i}(s))}{(\sum_{j=1}^n (\xi_{n;j}(t) - \xi_{n;j}(s))^2)^{1/2}}  
      (\xi_{n;i}(s) - \xi_{n;i}(r))^2 \Big( \sum_{j=1}^n (\xi_{n;j}(t) - \xi_{n;j}(s))^2 \Big)^{1/2}  \\
    & \leq \sqrt{r_n} \Big( \sum_{j=1}^n (\xi_{n;j}(t) - \xi_{n;j}(s))^2 \Big)^{1/2} \sum_{k=1}^n (\xi_{n;k}(s) - \xi_{n;k}(r))^2.
  \end{align*}
  Hence, Conditions~\eqref{eq:tight_wbs_1} and \eqref{eq:tight_wbs_3} are sufficient for bounds of these sums.
  It remains to inspect 
  \begin{align*}
  \max_{i \neq j} \E[ D^2_{n;i} D^2_{n;j} | \mac{C}_n] \leq &  \max_{i=1,\dots,n} \E[ D^2_{n;i} | \mac{C}_n]^2  \in \mac{O}_p(1), \\
   n \max_{i \neq j \neq k \neq i} | \E[ D^2_{n;i} D_{n;j} D_{n;k} | \mac{C}_n] |
    \leq & \max_{i,j=1,\dots,n} \E[ D^2_{n;i} | \mac{C}_n] n \E[ D_{n;j} | \mac{C}_n]^2 \in \mac{O}_p(1), \\
   n^2 \max_{i,j,k,l \text{ pairwise different}} | \E[ D_{n;i} D_{n;j} D_{n;k} D_{n;l} | \mac{C}_n] | 
    \leq & \max_{i=1,\dots, n} ( \sqrt{n} \E[ D_{n;i} | \mac{C}_n] )^4 \in \mac{O}_p(1).
  \end{align*}  
  It is also worth to mention that in fact a modified version of~\cite{billingsley99}, Theorem~13.5, is applied here
  such that the non-decreasing function therein may be replaced with a sequence of non-decreasing functions converging pointwise to a continuous one;
  see the remark in~\cite{jacod03}, p. 356.
  Since we are considering conditional expectations, this condition was translated into the convergence in probability in~\eqref{eq:tight_wbs_7}
  by applying the subsequence principle.
\end{proof}
\begin{proof}[Proof of Theorem~\ref{thm:ddmb_cifs}]
 The result follows from Lemmas~\ref{lemma:fidi_conv_general_wild_bs} and~\ref{lemma:tightness_wbs}
 taken into account Remark~\ref{rem:apply_lemmas} and the calculations in the proof of Theorem 2 in~\cite{beyersmann13} to see that $n r_n \in \mac{O}_p(1)$.
 Also, note that the condition
 $$ \max_{1 \leq i \leq 2n} | \E[ D^3_{2n;i} |\mac{A}_n] | \in \mac{O}_p(n)$$
 is already fulfilled by~\eqref{eq:conv_wbs_4} in combination with Jensen's inequality applied with $g : x \mapsto x^{4/3}$.
\end{proof}
\begin{proof}[Proof of Example~\ref{ex: proc}]
  Only \textnormal{(d)} needs to be proven. 
  The other examples are obviously special cases of the proposed DDMB of Theorem~\ref{thm:ddmb_cifs}.
  For the weird bootstrap, the limits of conditional mean and variance are given as
  \begin{align*}
      & \sqrt{n} | \E \left[ \left. B_i \right|\mac{A}_n \right] | 
	    = \sqrt{n} \Big(1 - Y(\tilde T_i) \frac{1}{Y(\Tilde T_i)} \Big) = 0 \\
      \text{and} \quad & | \var (B_i | \mac{A}_n ) - 1 |
	    = \Big| Y(\tilde T_i) \frac{1}{Y(\tilde T_i)} 
		\Big( 1 - \frac{1}{Y(\tilde T_i)} \Big) - 1 \Big|
	    \leq \sup_{s \in [0,t]} \frac{1}{Y(s)} \oPo 0
  \end{align*}
  and the convergence is due to Condition~\eqref{eq:cond_beyersmann}.
  Obviously, the Lyapunov condition in Remark~\ref{rem:thm_ddmb}(a) holds too
  and \eqref{eq:conv_wbs_5} holds per definition of the $B_i$. 
  Thus, we have shown that~\eqref{eq:conv_wbs_1} -- \eqref{eq:conv_wbs_6} are fulfilled.
\end{proof}


\bibliography{literatur}
\bibliographystyle{plainnat}

\end{document}